\documentclass[submission]{FPSAC2021_noheader}

\usepackage[utf8]{inputenc}
\newtheorem{thm}{Theorem}

\title[Multicritical random partitions]{Multicritical random partitions}

\author[D. Betea, J. Bouttier, H. Walsh]{Dan Betea\thanks{\href{mailto:dan.betea@gmail.com}{dan.betea@gmail.com}. D.B. is partially supported by FWO Flanders project EOS 30889451.}\addressmark{1}, J\'er\'emie Bouttier\thanks{\href{mailto:jeremie.bouttier@ipht.fr}{jeremie.bouttier@ipht.fr}. J.B. is partially supported by the project ANR-18-CE40-0033 ``Dimers''.}\addressmark{2,3}, \and Harriet Walsh\thanks{\href{mailto:harriet.walsh@ens-lyon.fr}{harriet.walsh@ens-lyon.fr}. This project has received funding from the European Research Council (ERC) under the European Union’s Horizon 2020 research and innovation programme (grant agreement No. ERC-2016-STG 716083, "CombiTop").}\addressmark{3,4}}

\address{\addressmark{1} Department of Mathematics, KU Leuven, Belgium \\ \addressmark{2} Université Paris-Saclay, CNRS, CEA, Institut de physique théorique, 91191 Gif-sur-Yvette, France\\ \addressmark{3} Univ Lyon, ENS de Lyon, Univ Claude Bernard, CNRS, Laboratoire de Physique, F-69342 Lyon\\ \addressmark{4} Université de Paris, CNRS, IRIF, F-75006, Paris, France}

\received{\today}

\abstract{We study two families of probability measures on integer
  partitions, which are Schur measures with parameters tuned in such a
  way that the edge fluctuations are characterized by a critical
  exponent different from the generic $1/3$. We find that the first
  part asymptotically follows  a ``higher-order analogue'' of the
  Tracy--Widom GUE distribution, previously
  encountered by Le Doussal, Majumdar and Schehr in quantum
  statistical physics. We also compute limit shapes, and discuss an
  exact mapping between one of our families and the multicritical
  unitary matrix models introduced by Periwal and Shevitz.}

\resume{Nous considérons deux familles de mesures de Schur dont les
  fluctuations de bord sont caractérisées par un exposant différant de
  la valeur générique $1/3$. Les distributions-limites, généralisant
  la loi de Tracy-Widom, ont été précédemment rencontrées par Le
  Doussal, Majumdar et Schehr. Nous calculons les formes-limites et
  discutons du lien avec les modèles de matrices unitaires de Periwal
  et Shevitz.}

\usepackage[backend=bibtex]{biblatex}
\newtheorem{prop}{Proposition}
\newtheorem{lem}[prop]{Lemma}
\theoremstyle{definition}
\newtheorem{rem}[prop]{Remark}
\newtheorem{ex}[prop]{Example}

\usepackage{amsfonts, graphicx, psfrag}

\def\Z{\mathbb{Z}}

\def\R{\mathbb{R}}
\def\P{\mathbb{P}}

\def\E{\mathbb{E}}

\def\Ai{\mathrm{Ai}}
\def\mcA{\mathcal{A}}
\def\oe{\mathrm{oe}}
\def\o{\mathrm{o}}
\def\tr{\mathrm{tr}}
\def\dx{\mathrm{d}}

\newcommand{\smallbin}[2]{\left(\begin{smallmatrix}#1 \\ #2\end{smallmatrix}\right)}

\begin{document}

\maketitle

\section{Introduction}

\paragraph{Background.} An \emph{integer partition}, hereafter called
\emph{partition} for short, is a nonincreasing sequence
$\lambda = (\lambda_1 \geq \lambda_2 \geq \dots \geq 0)$ of
nonnegative integers which is eventually zero.  Its \emph{size} is
$|\lambda| := \sum_i \lambda_i$. The \emph{conjugate partition}
$\lambda'$, given by $\lambda'_j := |\{ i:\lambda_i \geq j \}|$, has
the same size as $\lambda$, and in particular $\lambda'_1$ is equal to
the number of nonzero elements of $\lambda$.

Schur measures, introduced by Okounkov~\cite{oko01}, are probability
measures on integer partitions of the form
\begin{equation} \label{eq:schur_measure_def}
  \P(\lambda) = Z^{-1} s_\lambda [\theta_1,\theta_2,\dots] s_\lambda [\theta'_1,\theta'_2,\dots].
\end{equation}
Here, the $\theta_i,\theta'_i$ are numbers such that
$Z = \exp \sum_{i\geq 1} \frac{\theta_i \theta'_i}{i}$ is well-defined, and
$s_\lambda [\theta_1,\theta_2,\dots]$ is the Schur symmetric function
indexed by $\lambda$ and evaluated at the specialization sending the
$i$-th power sum $p_i$ to the value $\theta_i$, for all $i \geq 1$. A
more concrete expression is given by the Jacobi--Trudi identity
$s_\lambda[\theta_1,\theta_2,\dots] = \det_{i,j}
h_{\lambda_i-i+j}[\theta_1,\theta_2,\dots]$, the entries of the
determinant being given by the generating series
$\sum_{k \geq 0}
h_k[\theta_1,\theta_2,\dots] z^k = \exp \sum_{i \geq 1}
\tfrac{\theta_i z^i}{i}$. See~\cite{mac} for background on
symmetric functions and specializations.

\begin{ex}
  \label{ex:ppm}
  For $\theta_1=\theta_1'=\theta$, and all other $\theta_i,\theta_i'$
  set to zero, we obtain the \emph{poissonized Plancherel measure}
  $\P(\lambda) = e^{-\theta^2} \left(\theta^{|\lambda|}
  \frac{f_\lambda}{|\lambda|!}\right)^2$, discussed below. Here, $f_\lambda$
  denotes the number of standard Young tableaux of shape $\lambda$.
\end{ex}

\begin{ex}
  \label{ex:meastwo}
  For $\theta_1=\theta_1'$, $\theta_2=\theta_2'$, and all other
  $\theta_i,\theta_i'$ set to zero, we get
  \begin{equation}
    \P(\lambda) = e^{-\theta_1^2-\theta_2^2/2} \sum_{\mu = 1^{a_1} 2^{a_2}} \sum_{\nu = 1^{b_1} 2^{b_2}} \frac{ \chi^\lambda (\mu) \chi^\lambda (\nu) \theta_1^{a_1+b_1} \theta_2^{a_2+b_2}}{2^{a_2+b_2} a_1! a_2! b_1! b_2!} 
  \end{equation}
  where $\chi^{\lambda}$ is the irreducible character of the symmetric group $S_{|\lambda|}$ indexed by $\lambda$ and $\mu, \nu$ are two-column partitions with $|\lambda| = |\mu| = |\nu|$.
\end{ex}

Schur measures and their generalizations appear in several combinatorial, probabilistic, and statistical mechanical models of mathematical and physical interest. For a brief list, see~\cite{oko01, oko03, bg16} and references therein. One notable instance is the resolution of Ulam's problem
on longest increasing subsequence of random permutations.
Namely, if we consider the poissonized Plancherel measure in
Example~\ref{ex:ppm}, then the Baik--Deift--Johansson theorem~\cite{bdj99} states that
the first part $\lambda_1$ satisfies
\begin{equation} \label{eq:bdj}
    \lim_{\theta \to \infty} \P \left[ \frac{\lambda_1 - 2 \theta}{\theta^{1/3}} < s \right] = F_{\rm TW} (s)
\end{equation}
with $F_{\rm TW} (s)$ the Tracy--Widom GUE distribution~\cite{tw94_airy} from random matrix theory. By Schensted's theorem~\cite{sch61}, $\lambda_1$ is equal in distribution to the longest increasing subsequence of a random permutation on $S_N$, the symmetric group of $N$ letters, where $N$ in our case is a Poisson random variable $N \sim {\rm Poisson}(\theta^2)$. See~\cite{rom15} for more on this topic.

\paragraph{Main contribution.} We consider \emph{multicritical} Schur measures, having as their salient
feature an ``edge'' behavior different from~\eqref{eq:bdj}. More
precisely, for every $n \geq 2$, we construct Schur measures for which
the $1/3$ fluctuation exponent is replaced by $1/(2n+1)$ (we recover
the poissonized Plancherel measure for $n=1$).  The limiting
distribution then becomes a ``higher-order analogue'' of the
Tracy--Widom distribution. It is a $\tau$-function of a higher-order
differential equation of the Painlev\'e II hierarchy~\cite{ccg19} in the same way the Tracy--Widom distribution is for the ``classical''
Painlev\'e II equation~\cite{tw94_airy}.

Our inspiration comes from the work of Le Doussal, Majumdar and
Schehr~\cite{ldms18}, who found the same limiting distributions in the
momenta statistics of fermions in nonharmonic traps. They also noted a
coincidental connection with the multicritical unitary matrix models
of Periwal and Shevitz~\cite{ps90}, which involve the Painlevé II
hierarchy in their double scaling limit.

Our multicritical Schur measures explain the origin of this
connection. On the one hand, as observed by Okounkov~\cite{oko01},
Schur measures admit a convenient description in terms of free
fermions. Simple scaling arguments show that they have the same
asymptotic edge behavior as the models considered in~\cite{ldms18}. On
the other hand, through a chain of classical identities that we will
review, the distribution of $\lambda_1$ in a Schur measure can be
expressed as the partition function of a unitary matrix model. For
multicritical measures, we recover \emph{exactly} the models
of~\cite{ps90}. Let us point out that there is a known connection between Ulam's
problem and the Gross--Witten unitary matrix model, see~\cite{joh98}
and references therein. We comment on the relation with our work in
the conclusion.

\paragraph{Outline.}  In Section~\ref{sec:edge}, we define the
multicritical Schur measures and state our main theorems
(Theorems~\ref{thm:edge_oe} and~\ref{thm:edge_o}) concerning their
edge behavior. We compute limit shapes in
Section~\ref{sec:ls}. Section~\ref{sec:toep} reviews the connection
between Schur measures and unitary matrix integrals. The proof of
Theorems~\ref{thm:edge_oe} and~\ref{thm:edge_o} is sketched in
Section~\ref{sec:proof}. Finally, concluding remarks are gathered in
Section~\ref{sec:conc}.

This is an extended abstract of the paper~\cite{BBWlong}. For brevity,
we do not include a discussion of the physical interpretation in terms
of fermions here, but we note that they manifest themselves via the
determinantal point processes used in Section~\ref{sec:proof}.

\section{Multicritical Schur measures and their edge behavior}
\label{sec:edge}

A partition~$\lambda$ may be characterized by the set
$S(\lambda)=\{ \lambda_i - i + \frac{1}{2} | i \geq 1 \} \subset
\Z+\frac12$, see Figure~\ref{fig:ydtomaya} below. Note that the
largest element of $S(\lambda)$ is $\lambda_1-\frac12$, and the
smallest element of its complement is $-\lambda_1'+\frac12$, since
$-S(\lambda)$ is the complement of $S(\lambda')$.  When $\lambda$ is a
distributed according to a Schur measure~\eqref{eq:schur_measure_def},
it was shown by Okounkov~\cite{oko01} that $S(\lambda)$ is a
determinantal point process, whose kernel admits an explicit
expression (given in Section~\ref{sec:proof}) in terms of the
$\theta_i,\theta_i'$ parameters.

The study of the edge behavior---the statistics of the largest
element(s) of $S(\lambda)$, or of the smallest element(s) of its
complement---is most conveniently done via a saddle-point
analysis~\cite{oko03}. For generic parameters $\theta_i,\theta_i'$
(and, in particular, for the poissonized Plancherel measure), it is
found that the edge behavior is characterized by the coalescence of
two saddle points, which implies that the ``action'' has a double
critical point, also known as ``monkey saddle'', explaining the $1/3$
fluctuation exponent. Multicritical Schur measures are obtained by
tuning the parameters in such a way that the action has a critical
point of higher order.

For simplicity, we restrict to the case where
$\theta_i=\theta_i'$---ensuring that the
probability~\eqref{eq:schur_measure_def} is indeed nonnegative---and
where the set $\{i: \theta_i \neq 0\}$ is finite and of fixed cardinal
$n \geq 1$. By symmetry reasons, the edge critical point is always of
even order, and by tuning the $\theta_i$ we expect $2n$ to be the
maximal possible order. This is indeed the case.

\begin{thm}[``odd-even multicritical measure''] \label{thm:edge_oe} Let
  $\P^\oe_{n,\theta}$ denote the Schur
  measure~\eqref{eq:schur_measure_def} where we set
  $\theta_{i} = \frac{(-1)^{i+1} (n-1)! (n+1)!}{(n-i)! (n+i)!} \theta$
  for $i=1, \dots, n$, and $\theta_i=0$ for $i>n$. Then, we have
  \begin{equation}
    \label{eq:edge_oe}
      \lim_{\theta \to \infty} \P^\oe_{n,\theta} \left[ \frac{\lambda_1 - b \theta}{ (\theta d)^{\frac{1}{2n+1}} } < s \right] = F (2n+1; s), \quad  \lim_{\theta \to \infty} \P^\oe_{n,\theta} \left[ \frac{\lambda'_1 - \tilde{b} \theta}{ (\theta \tilde{d})^{\frac{1}{3}} } < s \right] = F (3; s)
  \end{equation}
  with $b = \frac{n+1}{n}$, $d = \binom{2n}{n-1}$,
  $\tilde{b} = \frac{n+1}{n} \left(\frac{(2n)!!}{(2n-1)!!} -
    1\right)$, $\tilde{d} = 2^{2n-2} n \binom{2n}{n-1}^{-1}$,
  $F(3;s)=F_{\rm TW}(s)$ the Tracy-Widom GUE distribution and
  $F(2n+1,s)$ its higher-order analogue defined in~\eqref{eq:F_def}.
\end{thm}

As we see, we obtain a nongeneric exponent $1/(2n+1)$ for the
fluctuations of $\lambda_1$, but we still have the generic exponent
$1/3$ for the fluctuations of $\lambda'_1$. It is actually possible
to have a more symmetric situation if, rather than taking
$\theta_1,\ldots,\theta_n$ nonzero, we take
$\theta_1,\theta_3,\ldots,\theta_{2n-1}$ nonzero.

\begin{thm}[``odd multicritical measure''] \label{thm:edge_o}
    Let $\P^\o_{n,\theta}$ denote the Schur measure~\eqref{eq:schur_measure_def} where we set $\theta_{2i-1} = \frac{(-1)^{i+1} (n-1)! n!}{(2i-1) (n-i)! (n+i-1)!} \theta$ for $i=1, \dots, n$, and all other $\theta_i$ to zero.
    Then, $\P^\o_{n,\theta}$ is invariant under the conjugation of partitions $\lambda \mapsto \lambda'$, and we have
    \begin{equation}
        \lim_{\theta \to \infty} \P^\o_{n,\theta} \left[ \frac{\lambda_1 - b \theta}{ (\theta d)^{\frac{1}{2n+1}} } < s \right] = \lim_{\theta \to \infty} \P^\o_{n,\theta} \left[ \frac{\lambda'_1 - b \theta}{ (\theta d)^{\frac{1}{2n+1}} } < s \right] = F (2n+1; s) 
    \end{equation}
    with $b = 2^{4n-1}n^{-1} \binom{2n}{n}^{-2}$, $d = \frac{(2n-1)!!}{(2n-2)!!}$, and $F(2n+1;s)$ defined at~\eqref{eq:F_def} below.
  \end{thm}

\begin{rem}
  For both measures, we have $\theta_1=\theta$ and the parameters $\theta_i$, $b$ and $d$ satisfy
    \begin{equation} \label{eq:multicrit}
        2\sum_{i} i^k \theta_i = \delta_{k, 0} \, b \theta + \delta_{k,2n} (-1)^{n+1} (2n)!  d \theta, \quad k=0, 2,\dots , 2n-2, 2n.
    \end{equation}
\end{rem}

When $n=1$, both measures reduce to the poissonized Plancherel
measure, and we recover the convergence in
distribution~\eqref{eq:bdj}. As soon as $n \geq 2$, they involve
specializations which are not Schur positive, but the measures are
nevertheless probability measures.

\begin{ex}
  For $n=2$, $\P^\oe_{n,\theta}$ has the form given in
  Example~\ref{ex:meastwo} with $\theta_1=\theta$,
  $\theta_2=-\frac{\theta}4$, while $\P^\o_{n,\theta}$ has
  $\theta_1=\theta$, $\theta_3=-\frac{\theta}9$ as nonzero parameters.
\end{ex}

The distributions $F(2n+1;s)$ appearing in Theorems~\ref{thm:edge_oe}
and \ref{thm:edge_o} have been previously encountered
in~\cite{ldms18,ccg19}, and we now give their definition in a
self-contained way.  First, we recall that, if $K$ is an integral
operator with kernel $K(x,y)$ acting on $L^2(X)$ ($X$ is an open
interval in what follows), it acts on functions $f \in L^2(X)$ via
``matrix multiplication'' $(Kf) (x) = \int_X K(x, y) f(y) \dx y$. For
such operators which are trace-class one can define the \emph{Fredholm
  determinant} of $1-K$ ($1$ the identity operator) on $L^2 (X)$ by
\begin{equation}
    \det(1-K)_{L^2(X)} = \sum_{m \geq 0} \frac{(-1)^m}{m!} \int_X \cdots \int_X \det_{1 \leq i, j \leq m} [K(x_i, x_j)] \dx x_1 \cdots \dx x_m
\end{equation}
where there are $m$ integrals in the $m$-th summand (and the term $m=0$ yields 1).

Consider the following \emph{generalized} (order $2n+1$) \emph{Airy function}:
\begin{equation}
    \Ai_{2n+1} (x) = \int_{i \R + \delta}  \exp \left( \frac{(-1)^{n-1} \zeta^{2n+1}}{2n+1} - x \zeta \right) \frac{\dx \zeta}{2 \pi i}
\end{equation}
where $\delta>0$ is small and the contour is up-oriented.\footnote{Comparing with~\cite[Eq.~(5)]{ldms18}, we chose different integration conventions for the same function. Their expression is different for $n$ even and comes from the change of variables $z = - \zeta$. Otherwise said, the contours of~\cite[Eq.~(5)]{ldms18} are such that $\Re(z^{2n+1}) < 0$ whereas ours have $\Re((-1)^{n-1} \zeta^{2n+1}) < 0$.} Notice they satisfy the generalized Airy differential equation $\left(\frac{d}{dx}\right)^{2n} A(x) = (-1)^{n-1} x A(x)$ and that $\Ai_3$ is the usual Airy \emph{Ai} function. Then $F(2n+1; s)$ is the following Fredholm determinant
\begin{equation} \label{eq:F_def}
    F(2n+1; s) = \det(1-\mcA_{2n+1})_{L^2(s, \infty)}
\end{equation}
where $\mcA_{2n+1}$ is the higher order Airy kernel given by
\begin{multline} \label{eq:def_airy_gen}
    \mcA_{2n+1} (x, y) = \int\limits_{i \R - \delta} \frac{\dx \omega}{2 \pi i} \int\limits_{i \R + \delta} \frac{\dx \zeta}{2 \pi i} \frac{\exp \left( \frac{(-1)^{n-1} \zeta^{2n+1}}{2n+1} - x \zeta \right)}{\exp \left( \frac{(-1)^{n-1} \omega^{2n+1}}{2n+1} - y \omega \right)}  \frac{1}{\zeta - \omega} \\
    = \int_0^\infty \Ai_{2n+1} (x+t) \Ai_{2n+1} (y+t) \dx t
    =  \frac{\sum_{i=0}^{2n-1} (-1)^{n-1+i} \Ai_{2n+1}^{(i)} (x) \Ai_{2n+1}^{(2m-1-i)} (y) }{x-y}
\end{multline}
(both contours above are up-oriented). Note that $\mcA_3(x, y) = \frac{\Ai_3(x) \Ai'_3(y) - \Ai'_3(x) \Ai_3(y)}{x-y}$
is the usual Airy kernel and that $F(3; s) = F_{\rm TW}(s)$ is the Tracy--Widom GUE distribution~\cite{tw94_airy}. In the $x=y$ case, the third equality should be taken in the l'H\^opital limit sense.

\section{Limit shapes} \label{sec:ls}

\begin{figure}
  \centering
  \def\svgwidth{0.5\columnwidth}
  {\scriptsize 
\begingroup%
  \makeatletter%
  \providecommand\color[2][]{%
    \errmessage{(Inkscape) Color is used for the text in Inkscape, but the package 'color.sty' is not loaded}%
    \renewcommand\color[2][]{}%
  }%
  \providecommand\transparent[1]{%
    \errmessage{(Inkscape) Transparency is used (non-zero) for the text in Inkscape, but the package 'transparent.sty' is not loaded}%
    \renewcommand\transparent[1]{}%
  }%
  \providecommand\rotatebox[2]{#2}%
  \newcommand*\fsize{\dimexpr\f@size pt\relax}%
  \newcommand*\lineheight[1]{\fontsize{\fsize}{#1\fsize}\selectfont}%
  \ifx\svgwidth\undefined%
    \setlength{\unitlength}{421.79745483bp}%
    \ifx\svgscale\undefined%
      \relax%
    \else%
      \setlength{\unitlength}{\unitlength * \real{\svgscale}}%
    \fi%
  \else%
    \setlength{\unitlength}{\svgwidth}%
  \fi%
  \global\let\svgwidth\undefined%
  \global\let\svgscale\undefined%
  \makeatother%
  \begin{picture}(1,0.62593033)%
    \lineheight{1}%
    \setlength\tabcolsep{0pt}%
    \put(0,0){\includegraphics[width=\unitlength]{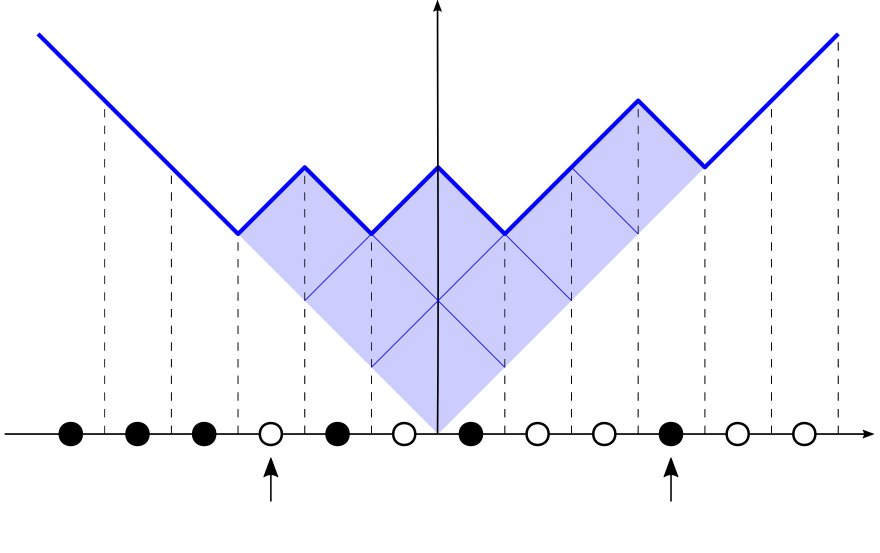}}%
    \put(0.25798842,0.03068229){\makebox(0,0)[lt]{\lineheight{1.25}\smash{\begin{tabular}[t]{l}left edge\end{tabular}}}}%
    \put(0.69776361,0.03074017){\makebox(0,0)[lt]{\lineheight{1.25}\smash{\begin{tabular}[t]{l}right edge\end{tabular}}}}%
  \end{picture}%
\endgroup%
}
  \caption{The profile (thick blue line) and the set $S(\lambda)=\{ \frac{7}{2},\frac{1}{2},-\frac{3}{2},-\frac{7}{2}, -\frac{9}{2},\ldots \} $ (black dots, corresponding to the $-1$ slopes in the profile) for the partition $\lambda=(4,2,1)$.}
  \label{fig:ydtomaya}
\end{figure}

In this section we describe the limit shapes for the multicritical $\P^\o_{n,\theta}$- and $\P^\oe_{n,\theta}$-distributed random partitions of Theorems~\ref{thm:edge_oe} and~\ref{thm:edge_o}. Proofs are omitted for brevity.

To begin, recall that the Young diagram of a partition can be
represented in ``Russian convention'' as the graph of a piecewise
linear function composed of slope $\pm 1$ segments, which we call its
\emph{profile}. See Figure~\ref{fig:ydtomaya}.

If $\lambda$ is distributed according to the measures $\P^\oe_{n,\theta}$ or
$\P^\o_{n,\theta}$ of Theorems~\ref{thm:edge_oe} and~\ref{thm:edge_o}, and if we
rescale by a factor $1/\sqrt{\theta}$ in both directions, then the
profile converges as $\theta \to \infty$ to the graph of a
deterministic $1$-Lipschitz function, denoted $\Omega = \Omega^{\o / \oe}_n$. We have
$\Omega'=1-2\rho$, where $\rho$ is the limiting density profile of the set $S(\lambda)=\{ \lambda_i - i + \frac{1}{2} | i \geq 1 \}$.

\begin{figure}
  \centering
  \def\svgwidth{0.5\columnwidth}
  {\scriptsize
  \input{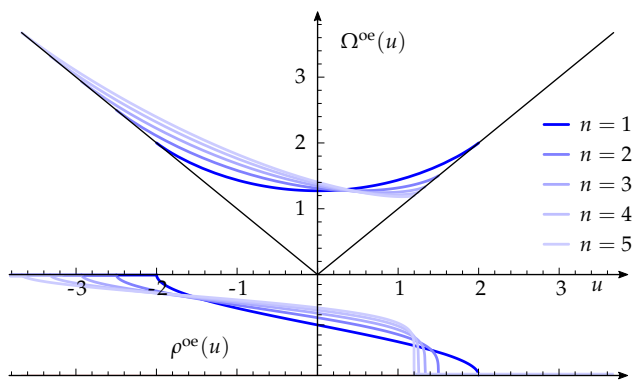}
  }
  \caption{Limit shape and density profile of $\P^\oe_{n,\theta}$-distributed
    random partitions.}
  \label{fig:asymetric}
\end{figure}

\begin{figure}
  \centering
  \def\svgwidth{0.5\columnwidth}
  {\scriptsize
  \input{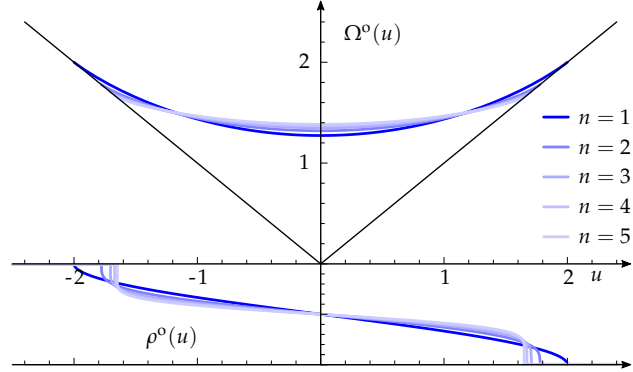}
  }
  \caption{Limit shape and density profile of $\P^\o_{n,\theta}$-distributed
    random partitions, for $n=1,\dots,5$. Notice the symmetry with respect to the vertical axis.}
  \label{fig:symetric}
\end{figure}

The limiting density profiles may be computed exactly. Let us denote them as follows: 
\begin{equation}
  \rho^{\o / \oe}_n(u) = \lim_{\theta \to \infty} \sum_{\lambda : \theta u \in S(\lambda)} \mathbb{P}^{\o / \oe}_{n,\theta} (\lambda).
\end{equation}
In the $\oe$ case we have, with $b = \frac{n+1}{n}$, $\tilde{b} = \frac{n+1}{n} \left(\frac{(2n)!!}{(2n-1)!!}-1\right)$:
\begin{equation}
  \rho^\oe_n(u) =   \tfrac{1}{\pi}\arccos \left( 1 - \tfrac{1}{2}\smallbin{2n}{n-1}^{\frac{1}{n}}( b-u \big)^{\frac{1}{n}}\right), \quad u \in \left[-\tilde{b},b\right]
\end{equation}
and $\rho^\oe_n(u) = 1$ for $u < -\tilde{b}$, $\rho^\oe_n(u) = 0$ for $u > b$. The limit profile---depicted in Figure~\ref{fig:asymetric}---is 
$\Omega^{\oe}_n(u) = \tilde{b} + \int_{-\tilde{b}}^{u} \left[ 1- 2 \rho^\oe_n(v)\right] \dx v$.  
A similar profile, for $n$=2, was recently observed in tight-binding fermions~\cite{Bocini_Stephan_2020}. 

In the $\o$ case and for $b = 2^{4n-1} n^{-1} \smallbin{2n}{n}^{-2}$ we have: 
\begin{equation}
\rho^\o_n(u) =  \frac{\chi(u)}{\pi}, \quad \int^{\chi(u)}_0 (2 \sin \phi)^{2n-1} \dx \phi = (-1)^{n+1}\smallbin{2n-1}{n} u,\quad u \in \left[-b,b\right]
\end{equation} 
continued to $\rho^\o_n(u) = 1$ for $u < -{b}$ and $\rho^\o_n(u) = 0$ for $u > b$. The limit shape, \emph{symmetric under the vertical axis} and shown in Figure~\ref{fig:symetric}, is $\Omega^{\o}_n(u) = b + \int_{-b}^{u} \left[ 1- 2 \rho^\o_n(v)\right] \dx v$.

Both $\Omega^\o_n$ and $\Omega^\oe_n$ are extensions of the Vershik--Kerov--Logan--Shepp limit curve---see e.g.~\cite{rom15}---to multicritical random partitions; indeed they become the former if $n=1$.

\section{Toeplitz determinants and unitary matrix integrals}
\label{sec:toep}

In this section we review the connection between Schur measures and
unitary matrix integrals, and we relate our multicritical measures 
to the integrals studied in~\cite{ps90}. For simplicity, we assume that the parameters
$\theta_i,\theta'_i$ of~\eqref{eq:schur_measure_def} are such that
$\theta_i=\theta_i'$ for all $i$, and $\theta_i=0$ for $i$ large
enough. We introduce the polynomials $V$ and $\tilde{V}$ defined by
\begin{equation}
  \label{eq:Vdef}
    V(z) = \sum_{i \geq 1} \theta_i \frac{z^i}i, \quad \tilde{V}(z+z^{-1}) = V(z) + V(z^{-1}).
\end{equation}
In physical parlance $\tilde{V}$, modulo a multiplicative constant, is often called the \emph{potential}.

\begin{ex}
    If $V(z) = \theta_1 z + \frac{\theta_2}{2} z^2 + \frac{\theta}{3} z^3$ we have $\tilde{V}(x) = -\theta_2 + (\theta_1-\theta_3) x + \frac{\theta_2 }{2} x^2 + \frac{\theta_3 }{3} x^3$.
\end{ex}

\begin{prop} \label{prop:unit_1}
  For $\lambda$ distributed as in~\eqref{eq:schur_measure_def} with $\theta_i = \theta'_i$ for all $i$, we have:
\begin{equation}
    e^{\sum_i \theta_i^2/i} \cdot \P [\lambda'_1 \leq \ell] = \det_{1 \leq i, j \leq \ell}  [f_{j-i}] = \E_{U \in \mathcal{U}(\ell)} \left[ \exp \tr\, \tilde{V}(U+U^*) \right]
  \end{equation}
  where the middle Toeplitz determinant has symbol
  $\sum_{k \in \Z} f_k z^k = \exp \tilde{V}(z+z^{-1})$, and $\E_{U \in \mathcal{U}(\ell)}$ is the expectation with respect to the Haar measure over the unitary group $\mathcal{U}(\ell)$.
\end{prop}

\begin{proof}
  The left-hand side is equal to
  $\sum_{\lambda'_1 \leq \ell}
  (s_\lambda[\theta_1,\theta_2,\ldots])^2$ which, by Gessel's
  identity~\cite[Thm.~16]{ges90}, is equal to the middle Toeplitz
  determinant. The second equality is Heine's identity. 
\end{proof}

We also have the following similar identity regarding $\lambda_1$.

\begin{prop} \label{prop:unit_2}
  For $\lambda$ distributed as in~\eqref{eq:schur_measure_def} with $\theta_i = \theta'_i$ for all $i$, we have:
    \begin{equation}
      \label{eq:unit_2}
    e^{\sum_i \theta_i^2/i} \cdot \P [\lambda_1 \leq \ell] = \det_{1 \leq i, j \leq \ell}  [g_{j-i}] = \E_{U \in \mathcal{U}(\ell)} \left[ \exp \tr\, (-\tilde{V}(-U-U^*)) \right]
  \end{equation}
  where the middle Toeplitz determinant has symbol
  $\sum_{k \in \Z} g_k z^k = \exp (-\tilde{V}(-z-z^{-1}))$.
\end{prop}

It is a straightforward consequence of Proposition~\ref{prop:unit_1} and the following:

\begin{lem}
  If $\lambda$ is distributed according to the Schur
  measure~\eqref{eq:schur_measure_def}, then the conjugate partition
  $\lambda'$ is distributed according to the Schur measure of
  parameters $\tilde{\theta_i}=(-1)^{i-1}\theta_i$,
  $\tilde{\theta_i}'=(-1)^{i-1}\theta_i'$.
\end{lem}

\begin{proof}
  This follows from the relation $s_\lambda[\theta_1,\theta_2,\ldots] = s_{\lambda'}[\tilde{\theta}_1,\tilde{\theta}_2,\ldots]$
  that results from the classical involution $\omega$ on the algebra
  of symmetric functions mapping the power sum $p_i$ to
  $(-1)^{i-1} p_i$ and the Schur function $s_\lambda$ to
  $s_{\lambda'}$.
\end{proof}

Another consequence of the above lemma is the fact, mentioned in
Theorem~\ref{thm:edge_o}, that $\P^\o$ is invariant under conjugation.

When we specialize Proposition~\ref{prop:unit_2} to the multicritical
measures $\P^\oe_{n,\theta}$ of Theorem~\ref{thm:edge_oe}, then the
right-hand side of~\eqref{eq:unit_2} matches, up to a change of
variable $U \to -U$, the multicritical unitary matrix integrals of
Periwal and Shevitz~\cite{ps90}. Indeed, the derivative $V'_k(z)$
given on p.~737 of \emph{op.\ cit.}\ is proportional to $V'(z)$ for
$k=n$ in our present notations, and the proportionality constant can
be reabsorbed in $\theta$.

\section{Sketch of proof} 
\label{sec:proof}

Let us sketch the proof of Theorems~\ref{thm:edge_oe}
and~\ref{thm:edge_o}. We present the argument for the
$\P^\o_{n,\theta}$ measure as it is slightly simpler, and make
comments at the end on the difference with the $\P^\oe_{n,\theta}$
measure.

We use the fact, already mentioned at the beginning of
Section~\ref{sec:edge}, that $S(\lambda)$ is a determinantal
point process. This means that, fixing $m$ and
$k_1, \dots, k_m \in \Z+\frac{1}{2}$, we have
\begin{equation}
    \P^\o_{n,\theta}(\{ k_1, \dots, k_m \} \in S(\lambda)) = \det_{1 \leq i, j \leq m} K(k_i, k_j) 
\end{equation}
where, by~\cite{oko01}, the discrete ($\ell^2$ operator) kernel $K$
equals (for some small $\epsilon > 0$)
\begin{equation} \label{eq:K_def}
    K(k, \ell) = \frac{1}{(2 \pi i )^2}\oint_{|w| = 1 - \epsilon} \oint_{|z| = 1 + \epsilon} \frac{ e^{V(z)-V(z^{-1}) }} { e^{V(w)-V(w^{-1}) } } \frac{\dx z \dx w}{z^{k+1/2} w^{-\ell+1/2} (z-w)}
\end{equation}
with $V$ as in~\eqref{eq:Vdef}.  Combinatorially, the above integral
is just coefficient extraction: we look at the coefficient of
$z^k/w^\ell$ in a generating series (since $\epsilon>0$,
$\frac1{z-w}$ should be expanded as $\sum_{i \geq 0} \frac{w^i}{z^{i+1}}$). Moreover, inclusion-exclusion---see e.g.~\cite[Sections 3 and 5]{bg16} or~\cite[Ch.~2]{rom15}---gives that the \emph{gap} probability $\P^\o_{n,\theta}(\lambda_1 \leq l)$ is equal to the discrete Fredholm determinant $\det(1-K)_{\ell^2 \{l+1/2,l+3/2,\dots\}}$.

In the multicritical regime we look for numbers $\beta$ and $\theta_1, \theta_3, \dots, \theta_{2n-1}$ satisfying
\begin{equation} \label{eq:thetabeta_relation}
    \sum_{i=1,3,\dots,2n-1} i^{k} \theta_i = - \delta_{k,0} \frac{\beta}{2}, \quad k = 0, 2, \dots, 2n-2
\end{equation} 
and solve for each of them in terms of $\theta_1=\theta$. We find 
$\beta=b \theta$ with $b$ and
$\theta_1, \theta_3, \dots, \theta_{2n-1}$ as in
Theorem~\ref{thm:edge_o}.  The correlation kernel becomes
\begin{equation} \label{eq:S_steepest_descent}
    K(k, \ell) = \frac{1}{(2 \pi i )^2}\oint_{|w| = 1 - \epsilon} \oint_{|z| = 1 + \epsilon} \frac{e^{\theta [S_0(z)-S_0(w)] } \dx z \dx w}{z^{k+1/2} w^{-\ell+1/2}(z-w)} 
\end{equation}
with $S_0(z) = \sum\limits_{i=1}^n \frac{(-1)^{i+1} (n-1)! n!}{(2i-1) (n-i)! (n+i-1)!} \frac{(z^{2i-1} - z^{1-2i})}{2i-1}$. The equations~\eqref{eq:thetabeta_relation} ensure that
\begin{equation}
  \label{eq:dervan}
    \left. (z \partial_z)^i [S_0(z) - b \log z] \right|_{z=1} = 0, \quad 1 \leq i \leq 2n
\end{equation}
meaning $z=1$ is a critical point of order $2n$. The same is true for $z=-1$. Notice that the relation~\eqref{eq:dervan} is automatically satisfied for even $i$ by the symmetry relation $S_0(z)+S_0(z^{-1})=0$; the specific choice of coefficients ensures that it also holds for odd $i$ between $1$ and $2n-1$.

We now analyze the scaling regime
\begin{equation}
  \label{eq:scalreg}
  \theta \to \infty, \qquad k = \lfloor b \theta + x (\theta d)^\frac{1}{2n+1} \rfloor,
  \qquad  \ell = \lfloor b \theta + y (\theta d)^\frac{1}{2n+1} \rfloor
\end{equation}
with $d = \frac{(2n-1)!!}{(2n-2)!!}$. In this regime, the
integral~\eqref{eq:S_steepest_descent} will be dominated by the
vicinity of the critical point $z=w=1$ (if we considered instead the
regime $k,\ell \approx -b\theta$, then the critical point $z=w=-1$
would dominate). We perform the change of variable
\begin{equation} \label{eq:coord_change}
    z = 1 + \zeta ( d \theta^{-1})^{\frac{1}{2n+1}}, \quad w = 1 + \omega ( d \theta^{-1})^{\frac{1}{2n+1}}
\end{equation}
where $\zeta$ and $\omega$ are to be integrated over $i\R+\delta$ and
$i\R-\delta$ respectively. The quantity $\theta S_0(z)-k \ln z$
which appears exponentiated in the integral may be approximated as
\begin{equation}
  \frac{S^{(2n+1)} (1)} {(2n+1)!} \frac{\zeta^{2n+1}}{d} - x \zeta + O\left(\frac{1}{\theta^{1/(2n+1)}}\right) = (-1)^{n+1} \frac{\zeta^{2n+1}}{2n+1} - x \zeta + O\left(\frac{1}{\theta^{1/(2n+1)}}\right)
\end{equation}
and we estimate $-\theta S_0(w)+\ell \ln w$ similarly. Plugging these
estimates into~\eqref{eq:S_steepest_descent}, we recognize the double
integral representation~\eqref{eq:def_airy_gen} for
$\mcA_{2n+1}(x, y)$. By analytical arguments (dominated convergence,
tail bounds, etc.) similar to those given in
e.g.~\cite[Section~5]{bb19}, we deduce that
\begin{equation}
    (d^{-1} \theta)^{\frac{1}{2n+1}} K \left( b \theta + x (d\theta)^{\frac{1}{2n+1}},\ b \theta + y (d\theta)^{\frac{1}{2n+1}} \right) \to \mcA_{2n+1}(x, y) \quad \text{as } \theta \to \infty .
\end{equation}

To finish the proof, we show that $K(k, \ell)$ has exponential decay which then shows the discrete Fredholm determinant $\P^\o_{n,\theta}(\lambda_1 \leq l) = \det (1-K)_{\ell^2\{ l+1/2, l+3/2, \dots \}}$ converges to the continuous one $\det (1-\mcA_{2n+1})_{L^2(s, \infty)} = F(2n+1; s)$ when $l = b \theta + s (d\theta)^{\frac{1}{2n+1}}$.  

In the odd+even multicritical case, the analysis of the scaling
regime~\eqref{eq:scalreg} is the same. However, we lose symmetry under
conjugation. This means that $V(z)$ and hence the function $S_0(z)$
appearing in~\eqref{eq:def_airy_gen} are not odd functions of $z$
anymore. At the point $z=-1$, which is relevant for studying the
asymptotics of $\lambda'_1$, we find that $S_0(z)+\tilde{b}\ln z$
has a \emph{generic} double critical point which
leads to the second equality in~\eqref{eq:edge_oe}.

\section{Concluding remarks}
\label{sec:conc}

In this paper we have introduced Schur measures displaying the same
multicritical edge behavior as the fermionic models considered by Le
Doussal, Majumdar and Schehr~\cite{ldms18}. We also computed limit
shapes and explained how our measures map exactly to the
Periwal--Shevitz multicritical unitary matrix models~\cite{ps90}. This
gives a combinatorial explanation to the coincidence noted
in~\cite{ldms18}.

The approach of Periwal and Shevitz relies on the method of orthogonal
polynomials. Through this approach, one obtains a different expression
for the higher order distributions~$F(2n+1;s)$ in terms of solutions
of the Painlevé II hierarchy. It is shown in \cite{ccg19}---see also
Appendix~G of the arXiv version of \cite{ldms18}---that it is indeed
equal to the Fredholm determinant~\eqref{eq:F_def}. Multicriticality
of a similar flavor was also observed at the spectrum edge of
\emph{Hermitian} random matrix ensembles by Claeys, Its and
Krasovsky~\cite{cik10}.

For $n=1$, our measures reduce to the poissonized Plancherel measure,
while on the unitary random matrix side we obtain a model first
studied by Gross and Witten, see e.g.~\cite{joh98} and references
therein. Our work shows that the connection observed by Johansson
in~\cite{joh98} extends to higher orders $n \geq 2$ of
multicriticality, even though the formulation in terms of longest
increasing subsequences seems more elusive.

It is interesting to consider multicritical analogues of other
random matrix limiting phenomena, e.g.~the Pearcey process. We plan 
to address this in future work, and expect that it involves coalescence
of an \emph{odd} number of critical points: 3 (for Pearcey), 5, 7, etc.

\paragraph{Main result of this note (odd case).} Let us summarize, in one place, the results of Theorem~\ref{thm:edge_o} on one hand and of Propositions~\ref{prop:unit_1} and~\ref{prop:unit_2} on the other.\footnote{An analogous result could be stated for Theorem~\ref{thm:edge_oe}; we omit it for brevity.} 

Fix $n \geq 1$ and let $\lambda$ be $\P^\o_{n,\theta}$-distributed~\eqref{eq:schur_measure_def} with $\theta_{2i-1} = \frac{(-1)^{i+1} (n-1)! n!}{(2i-1) (n-i)! (n+i-1)!} \theta$ for $i=1,\dots, n$ and $\theta > 0$. Let $b = 2^{4n-1}n^{-1} \binom{2n}{n}^{-2}$; $d = \frac{(2n-1)!!}{(2n-2)!!}$; $V(z) = \sum_{i=1}^n \frac{\theta_{2i-1} z^{2i-1}}{2i-1}$; $\tilde{V}(z+z^{-1}) = V(z)+V(z^{-1})$; and $\sum_{k \in \Z} f_k z^k = \exp [V(z) + V(z^{-1})]$. Then the quantities
    \begin{equation}
        \P^\o_{n,\theta} (\lambda_1 \leq \ell), \quad \P^\o_{n,\theta} (\lambda'_1 \leq \ell), \quad \frac{\det_{1 \leq i, j \leq \ell}  [f_{j-i}]} {e^{\sum_{i=1}^n \theta_{2i-1}^2 / (2i-1)} }, \quad \frac{\E_{U \in \mathcal{U}(\ell)} \left[ \exp \tr\, \tilde{V}(U+U^*) \right]} {e^{\sum_{i=1}^n \theta_{2i-1}^2 / (2i-1)} } 
    \end{equation}
are all equal, and equal to the Fredholm determinant $\det(1-K)$ ($K$ as in~\eqref{eq:K_def}) on $\{\ell+1/2, \ell+3/2, \dots \}$. Asymptotically, they tend to the distribution $F(2n+1; s)$ in~\eqref{eq:F_def} when $\ell = b \theta + s (\theta d)^\frac{1}{2n+1}$ and $\theta \to \infty$.

\acknowledgements{We thank S. Bocini, G. Chapuy, T. Claeys, A. Kuijlaars,
  P. Le Doussal, A. Occelli, G. Schehr and J.-M. Stéphan for support
  and conversations regarding this project.}

\printbibliography

\end{document}